\documentclass{article}

\newtheorem{theorem}{Theorem}

\newenvironment{proof}[1][Proof]{\noindent\textbf{#1.} }{\ \rule{0.5em}{0.5em}}
\input{tcilatex}
\begin{document}

\[
\text{\textbf{Proof of some conjectured formulas for }}\frac{\mathbf{1}}{%
\mathbf{\pi }}\text{ \textbf{by Z.W.Sun.}} 
\]%
\[
\text{Gert Almkvist and Alexander Aycock} 
\]

Recently Z.W.Sun found over hundred conjectured formulas for $\dfrac{1}{\pi }
$. Many of them were proved by H.H.Chan, J.Wan and W.Zudilin (see [3], [9]).
Here we show that several other formulas in [6] are simple transformations
of known formulas for $\dfrac{1}{\pi }$ , most of them due to Ramanujan.
E.g. the following monstrous formula (not in [6])

\[
\sum_{n=0}^{\infty }A_{n}P(n)\frac{1}{262537412640769728^{n}}=\frac{%
13803981511092062440689}{\pi \sqrt{163}} 
\]%
where%
\[
P(n)=4129922862271324476805+16564777691765267456000n 
\]%
and%
\[
A_{n}=1728^{n}\sum_{k=0}^{n}\dbinom{-1/12}{k}\dbinom{-7/12}{k}\dbinom{-5/12}{%
n-k}\dbinom{-11/12}{n-k} 
\]%
is a transformation of Chudnovsky's formula 
\[
\sum_{n=0}^{\infty }(-1)^{n}a_{n}(13591409+545140134n)\frac{1}{640320^{3n}}=%
\frac{53360\sqrt{640320}}{12\pi } 
\]%
where%
\[
a_{n}=\dbinom{2n}{n}\dbinom{3n}{n}\dbinom{6n}{3n} 
\]%
The transformation is%
\[
\sum_{n=0}^{\infty }A_{n}x^{n}=\frac{1}{\sqrt{1-1728x}}\sum_{n=0}^{\infty
}a_{n}\left( -\frac{x}{1-1728x}\right) ^{n} 
\]

\textbf{General Transformation.}

Assume that we have a Ramanujan-like formula 
\[
\sum_{n=0}^{\infty }a_{n}(a+bn)x_{0}^{n}=\frac{1}{\pi } 
\]%
We make the substitution%
\[
\sum_{n=0}^{\infty }A_{n}x^{n}=\frac{1}{\sqrt{1-Mx}}\sum_{n=0}^{\infty
}a_{n}\left( -\frac{x}{1-Mx}\right) ^{n} 
\]

\textbf{Proposition 1. }Let

\ 
\[
w_{0}=-\frac{x_{0}}{1-Mx_{0}} 
\]%
Then we have the formula%
\[
\sum_{n=0}^{\infty }A_{n}(A+Bn)w_{0}^{n}=\frac{1}{\pi } 
\]%
where%
\[
A=\left\{ \frac{1}{2}bMx_{0}+a(1-Mx_{0})\right\} (1-Mx_{0})^{-3/2} 
\]%
and%
\[
B=b(1-Mx_{0})^{-3/2} 
\]%
\newline

\textbf{Proof: }The transformation above is an involution, e.g. we also have%
\[
\sum_{n=0}^{\infty }a_{n}x^{n}=\frac{1}{\sqrt{1-Mx}}\sum_{n=0}^{\infty
}A_{n}\left( -\frac{x}{1-Mx}\right) ^{n} 
\]%
Let \ $\theta =x\dfrac{d}{dx}$ . Then%
\[
\sum_{n=0}^{\infty }a_{n}(a+bn)x^{n}=(a+b\theta )\sum_{n=0}^{\infty
}a_{n}x^{n} 
\]%
\[
=(a+b\theta )\left\{ \frac{1}{\sqrt{1-Mx}}\sum_{n=0}^{\infty }A_{n}\left( -%
\frac{x}{1-Mx}\right) ^{n}\right\} 
\]%
\[
=\sum_{n=0}^{\infty }(-1)^{n}A_{n}(a+b\theta )\left\{ \frac{x^{n}}{%
(1-Mx)^{n+1/2}}\right\} 
\]%
\[
=\sum_{n=0}^{\infty }A_{n}\left\{ \frac{a+bn}{(1-Mx)^{1/2}}+\frac{(n+\dfrac{1%
}{2})bx}{(1-Mx)^{3/2}}\right\} \left( -\frac{x}{1-Mx}\right) ^{n} 
\]%
Substituting \ $x=x_{0}$ \ we are done.

The example in the introduction is the case \ $s=\dfrac{1}{6}$ and \ $M=1728$
of the hypergeometric case%
\[
a_{n}=M^{n}\frac{(1/2)_{n}(s)_{n}(1-s)_{n}}{n!^{3}} 
\]

Proving the transformation is a Maple exercise in each of the cases \ $s=%
\dfrac{1}{3},\dfrac{1}{4},\dfrac{1}{6}$. E.g. in the case \ $s=\dfrac{1}{6}$
\ one shows that both sides satisfy the differential equation%
\[
y^{\prime \prime \prime }+\frac{3(1-3456x)}{x(1-1728x)}y^{\prime \prime }+%
\frac{1-11856x+20155392x^{2}}{x^{2}(1-1728x)^{2}}y^{\prime }-\frac{%
24(31-93312x)}{x^{2}(1-1728x)^{2}}=0 
\]%
and checks that the first four coefficients agree.

\textbf{The case }$\mathbf{s=}\dfrac{\mathbf{1}}{\mathbf{3}}.$

Here we have \ $M=108$ and 
\[
a_{n}=108^{n}\frac{(1/2)_{n}(1/3)_{n}(2/3)_{n}}{n!^{3}}=\dbinom{2n}{n}^{2}%
\dbinom{3n}{n} 
\]%
with%
\[
A_{n}=108^{n}\sum_{k=0}^{n}\dbinom{-2/3}{k}\dbinom{-1/6}{k}\dbinom{-1/3}{n-k}%
\dbinom{-5/6}{n-k} 
\]%
In the table below the fomula 
\[
\sum_{n=0}^{\infty }a_{n}(a+bn)x_{0}^{n}=\frac{1}{\pi } 
\]%
is transformed to%
\[
\sum_{n=0}^{\infty }A_{n}(A+Bn)w_{0}^{n}=\frac{1}{\pi } 
\]%
\[
\begin{tabular}{|l|l|l|l|l|l|l|}
\hline
\# in [6] & $x_{0}$ & $a$ & $b$ & $w_{0}$ & $A$ & $B$ \\ \hline
4.16 & $-\dfrac{1}{192}$ & $\dfrac{\sqrt{3}}{4}$ & $\dfrac{5\sqrt{3}}{4}$ & $%
\dfrac{1}{300}$ & $\dfrac{\sqrt{3}}{50}$ & $\dfrac{16\sqrt{3}}{25}$ \\ \hline
4.18 & $-\dfrac{1}{1728}$ & $\dfrac{7\sqrt{3}}{36}$ & $\dfrac{17\sqrt{3}}{12}
$ & $\dfrac{1}{1836}$ & $\dfrac{11\sqrt{51}}{306}$ & $\dfrac{48\sqrt{51}}{153%
}$ \\ \hline
4.20 & $-\dfrac{1}{8640}$ & $\dfrac{\sqrt{15}}{12}$ & $\dfrac{3\sqrt{15}}{4}$
& $\dfrac{1}{8748}$ & $\dfrac{85\sqrt{3}}{486}$ & $\dfrac{400\sqrt{3}}{243}$
\\ \hline
4.21 & $-\dfrac{1}{108\cdot 2^{10}}$ & $\dfrac{53\sqrt{3}}{288}$ & $\dfrac{%
205\sqrt{3}}{96}$ & $\dfrac{1}{110700}$ & $\dfrac{527\sqrt{123}}{18450}$ & $%
\dfrac{3072\sqrt{123}}{9225}$ \\ \hline
- & $-\dfrac{1}{108\cdot 3024}$ & $\dfrac{13\sqrt{7}}{108}$ & $\dfrac{55%
\sqrt{7}}{36}$ & $\dfrac{1}{326700}$ & $\dfrac{9989\sqrt{3}}{54450}$ & $%
\dfrac{127008\sqrt{3}}{54450}$ \\ \hline
4.22 & $-\dfrac{1}{108\cdot 500^{2}}$ & $\dfrac{827\sqrt{3}}{4500}$ & $%
\dfrac{4717\sqrt{3}}{1500}$ & $\dfrac{1}{27000108}$ & $\dfrac{97659\sqrt{267}%
}{4500018}$ & $\dfrac{1500000\sqrt{267}}{4500018}$ \\ \hline
4.17 & $\dfrac{1}{1458}$ & $\dfrac{8}{27}$ & $\dfrac{20}{9}$ & $-\dfrac{1}{%
1350}$ & $\dfrac{52\sqrt{3}}{225}$ & $\dfrac{36\sqrt{3}}{25}$ \\ \hline
4.19 & $\dfrac{1}{27\cdot 125}$ & $\dfrac{8\sqrt{3}}{45}$ & $\dfrac{22\sqrt{3%
}}{45}$ & $-\dfrac{1}{3267}$ & $\dfrac{100\sqrt{15}}{1089}$ & $\dfrac{250%
\sqrt{15}}{363}$ \\ \hline
4.14 & $-\dfrac{1}{27}$ & $\dfrac{4\sqrt{3}}{9}$ & $\dfrac{5\sqrt{3}}{3}$ & $%
\dfrac{1}{135}$ & $-\dfrac{2\sqrt{15}}{45}$ & $\dfrac{\sqrt{15}}{15}$ \\ 
\hline
\end{tabular}%
\]

The last formula from columns 2-4 is divergent but results in the following
supercongruence%
\[
\sum_{n=0}^{p-1}a_{n}(4+15n)\frac{1}{(-27)^{n}}\equiv 4p\QDOVERD( ) {-3}{p}%
\text{ }\func{mod}\text{ }p^{3}
\]%
conjectured by Z.W.Sun in [8].

\textbf{The case \ }$\mathbf{s=}\dfrac{\mathbf{1}}{\mathbf{4}}.$

Here we have \ $M=256$ \ and%
\[
a_{n}=256^{n}\frac{(1/2)_{n}(1/4)_{n}(3/4)_{n}}{n!^{3}}=\dbinom{2n}{n}^{2}%
\dbinom{4n}{2n}
\]%
with%
\[
A_{n}=256^{n}\sum_{k=0}^{n}\dbinom{-1/8}{k}\dbinom{-5/8}{k}\dbinom{-3/8}{n-k}%
\dbinom{-7/8}{n-k}
\]%
\[
\begin{tabular}{|l|l|l|l|l|l|l|}
\hline
\# in [6] & $x_{0}$ & $a$ & $b$ & $w_{0}$ & $A$ & $B$ \\ \hline
4.23 & $-\dfrac{1}{1024}$ & $\dfrac{3}{8}$ & $\dfrac{5}{2}$ & $\dfrac{1}{1280%
}$ & $\dfrac{\sqrt{5}}{20}$ & $\dfrac{4\sqrt{5}}{5}$ \\ \hline
- & $-\dfrac{1}{63^{2}}$ & $\dfrac{8\sqrt{7}}{63}$ & $\dfrac{65\sqrt{7}}{63}$
& $\dfrac{1}{4225}$ & $\dfrac{392\sqrt{7}}{4225}$ & $\dfrac{3969\sqrt{7}}{%
4225}$ \\ \hline
4.27 & $-\dfrac{1}{3\cdot 2^{12}}$ & $\dfrac{3\sqrt{3}}{16}$ & $\dfrac{7%
\sqrt{3}}{4}$ & $\dfrac{1}{12544}$ & $\dfrac{57}{196}$ & $\dfrac{144}{49}$
\\ \hline
4.29 & $-\dfrac{1}{288^{2}}$ & $\dfrac{23}{72}$ & $260\dfrac{65}{18}$ & $%
\dfrac{1}{83200}$ & $\dfrac{113\sqrt{13}}{1300}$ & $\dfrac{324\sqrt{13}}{325}
$ \\ \hline
- & $-\dfrac{1}{1280\cdot 72^{2}}$ & $\dfrac{41\sqrt{5}}{288}$ & $644\dfrac{%
161\sqrt{5}}{72}$ & $\dfrac{1}{6635776}$ & $\dfrac{32995}{103684}$ & $\dfrac{%
518400}{103684}$ \\ \hline
- & $-\dfrac{1}{14112^{2}}$ & $\dfrac{1123}{3528}$ & $\dfrac{5365}{882}$ & $%
\dfrac{1}{199148800}$ & $\dfrac{162833\sqrt{37}}{3111700}$ & $\dfrac{3111696%
\sqrt{37}}{3111700}$ \\ \hline
4.26 & $\dfrac{1}{648}$ & $\dfrac{2}{9}$ & $\dfrac{14}{9}$ & $-\dfrac{1}{392}
$ & $\dfrac{46}{49}$ & $\dfrac{162}{49}$ \\ \hline
4.25 & $\dfrac{1}{48^{2}}$ & $\dfrac{\sqrt{3}}{6}$ & $\dfrac{4\sqrt{3}}{3}$
& $-\dfrac{1}{2048}$ & $\dfrac{3\sqrt{3}}{16}$ & $\dfrac{9\sqrt{3}}{8}$ \\ 
\hline
4.28 & $\dfrac{1}{144^{2}}$ & $\dfrac{2\sqrt{2}}{9}$ & $\dfrac{20\sqrt{2}}{9}
$ & $-\dfrac{1}{20480}$ & $\dfrac{17\sqrt{10}}{80}$ & $\dfrac{81\sqrt{10}}{40%
}$ \\ \hline
4.30 & $\dfrac{1}{784^{2}}$ & $\dfrac{9\sqrt{3}}{49}$ & $\dfrac{120\sqrt{3}}{%
49}$ & $-\dfrac{1}{614400}$ & $\dfrac{361\sqrt{2}}{1600}$ & $\dfrac{2401%
\sqrt{2}}{800}$ \\ \hline
4.31 & $\dfrac{1}{16^{2}\cdot 99^{2}}$ & $\dfrac{19\sqrt{11}}{198}$ & $%
\dfrac{140\sqrt{11}}{99}$ & $-\dfrac{1}{2508800}$ & $\dfrac{1331\sqrt{22}}{%
19600}$ & $\dfrac{9801\sqrt{22}}{9800}$ \\ \hline
- & $\dfrac{1}{16^{2}\cdot 99^{4}}$ & $\dfrac{2206\sqrt{2}}{9801}$ & $\dfrac{%
52780\sqrt{2}}{9801}$ & $-\dfrac{1}{24591257600}$ & $\dfrac{8029841\sqrt{58}%
}{192119200}$ & $\dfrac{192119202\sqrt{58}}{192119200}$ \\ \hline
4.24 & $-\dfrac{1}{144}$ & $\dfrac{\sqrt{3}}{3}$ & $\dfrac{5\sqrt{3}}{3}$ & $%
\dfrac{1}{400}$ & $-\dfrac{3\sqrt{3}}{25}$ & $\dfrac{9\sqrt{3}}{25}$ \\ 
\hline
\end{tabular}%
\]

The last Ramanujan-like formula is divergent (for a proof see Guillera [2])
but leads to the conjectured supercongruence (already in [8])%
\[
\sum_{n=0}^{p-1}a_{n}(1+5n)\frac{1}{(-144)^{n}}\equiv p\QDOVERD( ) {-3}{p}%
\text{ }\func{mod}\text{ }p^{3}
\]

\textbf{The case \ }$\mathbf{s=}\dfrac{\mathbf{1}}{\mathbf{6}}.$

Here \ $M=1728$ and%
\[
a_{n}=1728^{n}\frac{(1/2)_{n}(1/6)_{n}(5/6)_{n}}{n!^{3}}=\dbinom{2n}{n}%
\dbinom{3n}{n}\dbinom{6n}{3n}
\]%
with%
\[
A_{n}=1728^{n}\sum_{k=0}^{n}\dbinom{-1/12}{k}\dbinom{-7/12}{k}\dbinom{-5/12}{%
n-k}\dbinom{-11/12}{n-k}
\]%
\[
\small
\begin{tabular}{|l|l|l|l|l|l|l|}
\hline
\# in [6] & $x_{0}$ & $a$ & $b$ & $w_{0}$ & $A$ & $B$ \\ \hline
4.33 & $-\dfrac{1}{15^{3}}$ & $\dfrac{8\sqrt{15}}{75}$ & $\dfrac{21\sqrt{15}%
}{25}$ & $\dfrac{1}{5103}$ & $-\dfrac{8\sqrt{7}}{189}$ & $\dfrac{125\sqrt{7}%
}{189}$ \\ \hline
4.36 & $-\dfrac{1}{32^{3}}$ & $\dfrac{15\sqrt{2}}{64}$ & $\dfrac{77\sqrt{2}}{%
32}$ & $\dfrac{1}{34496}$ & $\dfrac{39\sqrt{11}}{539}$ & $\dfrac{512\sqrt{11}%
}{539}$ \\ \hline
- & $-\dfrac{1}{3\cdot 160^{3}}$ & $\dfrac{93\sqrt{30}}{1600}$ & $\dfrac{759%
\sqrt{30}}{800}$ & $\dfrac{1}{12289728}$ & $\dfrac{11751\sqrt{3}}{64009}$ & $%
\dfrac{192000\sqrt{3}}{64009}$ \\ \hline
4.35 & $-\dfrac{1}{96^{3}}$ & $\dfrac{25\sqrt{6}}{192}$ & $\dfrac{57\sqrt{6}%
}{32}$ & $\dfrac{1}{886464}$ & $\dfrac{37\sqrt{19}}{513}$ & $\dfrac{512\sqrt{%
19}}{513}$ \\ \hline
- & $-\dfrac{1}{960^{3}}$ & $\dfrac{263\sqrt{15}}{3200}$ & $\dfrac{2709\sqrt{%
15}}{1600}$ & $\dfrac{1}{884737728}$ & $\dfrac{248853\sqrt{43}}{512001}$ & $%
\dfrac{512000\sqrt{43}}{512001}$ \\ \hline
- & $-\dfrac{1}{5280^{3}}$ & $\dfrac{10177\sqrt{330}}{580800}$ & $\dfrac{%
43617\sqrt{330}}{96800}$ & $\dfrac{1}{147197953728}$ & $\dfrac{3312613\sqrt{%
67}}{85184001}$ & $\dfrac{85184000\sqrt{67}}{85184001}$ \\ \hline
4.34 & $\dfrac{1}{20^{3}}$ & $\dfrac{3\sqrt{5}}{25}$ & $\dfrac{28\sqrt{5}}{25%
}$ & $-\dfrac{1}{6272}$ & $\dfrac{24\sqrt{2}}{49}$ & $\dfrac{125\sqrt{2}}{49}
$ \\ \hline
4.32 & $\dfrac{4}{60^{3}}$ & $\dfrac{2\sqrt{15}}{25}$ & $\dfrac{22\sqrt{15}}{%
25}$ & $-\dfrac{1}{52272}$ & $\dfrac{26}{121}$ & $\dfrac{250}{121}$ \\ \hline
4.37 & $\dfrac{1}{66^{3}}$ & $\dfrac{20\sqrt{33}}{363}$ & $252\dfrac{84\sqrt{%
33}}{121}$ & $-\dfrac{1}{285768}$ & $\dfrac{436}{1323}$ & $\dfrac{5324}{1323}
$ \\ \hline
- & $\dfrac{1}{255^{3}}$ & $\dfrac{144\sqrt{255}}{7225}$ & $\dfrac{2394\sqrt{%
255}}{7225}$ & $-\dfrac{1}{16579647}$ & $\dfrac{73936\sqrt{7}}{614061}$ & $%
\dfrac{1228250\sqrt{7}}{614061}$ \\ \hline
- & $-\dfrac{1}{640320^{3}}$ & $\frac{13591409\sqrt{10005}}{227897059584000}$
& $\frac{90856689\sqrt{10005}}{37982843264000}$ &  &  &  \\ \hline
\end{tabular}%
\]

We have deleted the formula obtained from Chudnovsky's formula since it is
in the Introduction.\ 

So far we have only considered Ramanujan series with rational \ $x_{0}$,
found in [1]. We give one example in case \ $s=\dfrac{1}{3}$ \ with \ $x_{0}=%
\dfrac{1}{72}(7\sqrt{3}-12)$

\[
\sum_{n=0}^{\infty }a_{n}(1+(5+\sqrt{3})n)\left( \frac{7\sqrt{3}-12}{72}%
\right) ^{n}=\frac{2+\sqrt{3}}{\pi } 
\]%
giving%
\[
\sum_{n=0}^{\infty }A_{n}\left\{ 2(27\sqrt{3}-41)+8(5+\sqrt{3})n\right\}
\left( \frac{15-14\sqrt{3}}{66^{2}}\right) ^{n}=\frac{254-134\sqrt{3}}{\pi } 
\]

In the paper [7] by Z.W.Sun there are some formulas for \ $\dfrac{1}{\pi }$
\ which are special cases of identities for the hypergeometric function%
\[
F(a,b;c;x)=\sum_{n=0}^{\infty }\frac{(a)_{n}(b)_{n}}{(c)_{n}n!}x^{n} 
\]%
Thus Theorem 1.1 (i) in [7] is the special cases \ $s=\dfrac{1}{2},\dfrac{1}{%
3},\dfrac{1}{4},\dfrac{1}{6}$ \ of

\textbf{Proposition 2:}

We have%
\[
\sum_{n=0}^{\infty }n(\frac{1}{2})^{n}\sum_{k=0}^{n}\frac{(s)_{k}(1-s)_{k}}{%
k!^{2}}\frac{(s)_{n-k}(1-s)_{n-k}}{(n-k)!^{2}}=\frac{2}{\pi }\sin (\pi s) 
\]

\textbf{Proof: }Let%
\[
f(x)=F(s,1-s;1;x)^{2} 
\]%
\[
=\sum_{n=0}^{\infty }x^{n}\sum_{k=0}^{n}\frac{(s)_{k}(1-s)_{k}}{k!^{2}}\frac{%
(s)_{n-k}(1-s)_{n-k}}{(n-k)!^{2}} 
\]%
Then%
\[
\theta f(x)=\sum_{n=0}^{\infty }nx^{n}\sum_{k=0}^{n}\frac{(s)_{k}(1-s)_{k}}{%
k!^{2}}\frac{(s)_{n-k}(1-s)_{n-k}}{(n-k)!^{2}} 
\]%
\[
=2s(1-s)x\cdot F(s,1-s;1;x)\cdot F(1+s,2-s;2;x) 
\]%
Put \ $x=\dfrac{1}{2}$ \ and use the evaluation%
\[
F(a,b;\frac{a+b+1}{2};\frac{1}{2})=\sqrt{\pi }\frac{\Gamma (\dfrac{a+b+1}{2})%
}{\Gamma (\dfrac{a+1}{2})\Gamma (\dfrac{b+1}{2})} 
\]%
We obtain%
\[
2\frac{1}{2}s(1-s)\pi \frac{\Gamma (1)}{\Gamma (\dfrac{1}{2}+\dfrac{s}{2}%
)\Gamma (1-\dfrac{s}{2})}\frac{\Gamma (2)}{\Gamma (1+\dfrac{s}{2})\Gamma (%
\dfrac{3}{2}-\dfrac{s}{2})} 
\]%
\[
=s(1-s)\pi \frac{1}{\Gamma (\dfrac{1}{2}-\dfrac{s}{2})\dfrac{s}{2}\Gamma (%
\dfrac{s}{2})}\frac{1}{\Gamma (\dfrac{1}{2}+\dfrac{s}{2})(\dfrac{1}{2}-%
\dfrac{s}{2})\Gamma (\dfrac{1}{2}-\dfrac{s}{2})} 
\]%
\[
=4\pi \frac{\sin (\dfrac{\pi s}{2})}{\pi }\frac{\sin (\pi (\dfrac{1}{2}-%
\dfrac{s}{2}))}{\pi }=\frac{4}{\pi }\sin (\frac{\pi s}{2})\cos (\frac{\pi s}{%
2})=\frac{2}{\pi }\sin (\pi s) 
\]

\[
\]

\textbf{Some other transformations.}

We start with proving Conjecture 4 in [6]. We have

\textbf{Proposition 3. }Let%
\[
A_{n}=\sum_{k=0}^{n}\dbinom{2k}{k}\dbinom{2n-2k}{n-k}\dbinom{-s}{k}\dbinom{%
-(1-s)}{n-k} 
\]%
Then the following formula is valid%
\[
\sum_{n=0}^{\infty }A_{n}x^{n}=\frac{1}{\sqrt{1+4x}}\text{ }_{3}F\text{ }%
_{2}(1/2,s,1-s;1,1;-\frac{4x^{2}}{1+4x}) 
\]

\textbf{Classical Proof:}

We first note the identities%
\[
\dbinom{2k}{k}=4^{k}\frac{(1/2)_{k}}{k!} 
\]%
\[
\dbinom{-s}{k}=(-1)^{k}\frac{(s)_{k}}{k!} 
\]%
where \ $(a)_{0}=1$ \ and%
\[
(a)_{k}=a(a+1)...(a+k-1)\text{ \ for \ }k>0 
\]%
We get%
\[
L=\sum_{n=0}^{\infty }\sum_{k=0}^{n}\dbinom{2k}{k}\dbinom{2n-2k}{n-k}\dbinom{%
-s}{k}\dbinom{-(1-s)}{n-k}x^{n} 
\]%
\[
=\sum_{n=0}^{\infty }\sum_{k=0}^{n}\frac{(1/2)_{k}(s)_{k}}{k!^{2}}\frac{%
(1/2)_{n-k}(1-s)_{n-k}}{(n-k)!^{2}}(-4x)^{n} 
\]%
\[
=F(1/2,s;1;-4x)F(1/2,1-s;1,-4x) 
\]%
Now Euler's identity%
\[
F(a,b;,c;x)=(1-x)^{c-a-b}F(c-a,c-b;c;x) 
\]%
leads to%
\[
F(1/2,1-s;1,-4x)=(1+4x)^{-1/2+s}F(1/2,s;1;-4x) 
\]%
Hence%
\[
L=\frac{1}{\sqrt{1+4x}}\left\{ (1+4x)^{s/2}F(1/2,s;1;-4x)\right\} ^{2} 
\]%
Now we use the following identity (see [2], p.176, Exercise 1b)%
\[
F(2a,b;2b;x)=(1-x)^{-a}F(a,b-a;b+1/2;\frac{x^{2}}{4(x-1)} 
\]%
with \ $a=s/2,b=1/2$ to get%
\[
F(s,1/2;1;-4x)=(1+4x)^{-s/2}F(s/2,(1-s)/2;1;-\frac{4x^{2}}{1+4x}) 
\]%
Finally Clausen's identity%
\[
F(a,b;a+b+1/2;x)^{2}=F(2a,2b,a+b;a+b+1/2,2a+2b;x) 
\]%
gives%
\[
L=\frac{1}{\sqrt{1+4x}}F(1/2,s,1-s;1,1;-\frac{4x^{2}}{1+4x}) 
\]%
and the proof is finished.

\textbf{Maple Proof:}

\textbf{\ }Using Maple one verifies that both sides satisfy the differential
equation%
\[
y^{\prime \prime \prime }+\frac{3(1+8x)}{x(1+4x)}y^{\prime \prime }+\frac{%
1+28x+(108+16s-16s^{2})x^{2}}{x^{2}(1+4x)^{2}}y^{\prime }+\frac{%
2(1+(6+8s-8s^{2})x)}{x^{2}(1+4x)^{2}}y=0 
\]%
Then we check that the first terms in the power series solutions agree.

\textbf{Proposition 4. }Let%
\[
a_{n}=\frac{(1/2)_{n}(s)_{n}(1-s)_{n}}{n!^{3}} 
\]%
Given a formula for \ $\dfrac{1}{\pi }$ of Ramanujan type%
\[
\sum_{n=0}^{\infty }a_{n}(a+bn)x_{0}^{n}=\frac{1}{\pi } 
\]%
Let%
\[
w_{0}=\frac{1}{2}(-x_{0}\pm \sqrt{x_{0}^{2}-x_{0}}) 
\]%
Then the transformation above gives the formulas%
\[
\sum_{n=0}^{\infty }A_{n}(A+Bn)w_{0}^{n}=\frac{1}{\pi } 
\]%
where%
\[
A=\sqrt{1+4w_{0}}\left\{ a+\frac{bw_{0}}{1+2w_{0}}\right\} 
\]%
\[
B=\frac{b(1+4w_{0})^{3/2}}{2(1+2w_{0})} 
\]

\textbf{Proof: }We have%
\[
\sum_{n=0}^{\infty }A_{n}w^{n}=\frac{1}{\sqrt{1+4w}}\sum_{n=0}^{\infty
}a_{n}(-\frac{w^{2}}{1+4w})^{n} 
\]%
Take \ $A+B\theta $ \ on both sides ( $\theta =w\dfrac{d}{dw}$ )%
\[
\sum_{n=0}^{\infty }A_{n}(A+Bn)w^{n}=\sum_{n=0}^{\infty }a_{n}\left\{ \frac{A%
}{\sqrt{1+4w}}+\frac{2B}{(1+4w)^{3/2}}(-w+(2w+1)n\right\} (-\frac{w^{2}}{1+4w%
})^{n} 
\]%
Now put \ $w=w_{0}$ \ so \ $-\dfrac{w_{0}^{2}}{1+4w_{0}}=x_{0}$ \ and the
right hand is \ $\sum_{n=0}^{\infty }a_{n}(a+bn)x_{0}^{n}.$ We get%
\[
a=\frac{A}{\sqrt{1+4w_{0}}}-\frac{2Bw_{0}}{(1+4w_{0})^{3/2}} 
\]%
\[
b=\frac{2B(2w_{0}+1)}{(1+4w_{0})^{3/2}} 
\]%
and solving for \ $A$ \ and \ $B$ \ we are done.

$s=\dfrac{1}{2}$%
\[
\begin{tabular}{|l|l|l|l|l|l|}
\hline
$x_{0}$ & $a$ & $b$ & $w_{0}$ & $A$ & $B$ \\ \hline
$-1$ & $\dfrac{1}{2}$ & $2$ & $\dfrac{1}{2}(1-\sqrt{2})$ & $\dfrac{-3+2\sqrt{%
2}}{2}$ & $\dfrac{-4+3\sqrt{2}}{2}$ \\ \hline
$-\dfrac{1}{8}$ & $\dfrac{\sqrt{2}}{4}$ & $\dfrac{3\sqrt{2}}{2}$ & $\dfrac{1%
}{16}(1\pm 3)$ & $\dfrac{1}{2}(1\pm 1)$ & $\dfrac{1}{4}(5\pm 3)$ \\ \hline
\end{tabular}%
\]

$s=\dfrac{1}{3}$%
\[
\small
\begin{tabular}{|l|l|l|l|l|l|}
\hline
$x_{0}$ & $a$ & $b$ & $w_{0}$ & $A$ & $B$ \\ \hline
$-\dfrac{9}{16}$ & $\dfrac{\sqrt{3}}{4}$ & $\dfrac{5\sqrt{3}}{4}$ & $\dfrac{3%
}{32}(3\pm 5)$ & $\dfrac{\sqrt{3}}{32}(19\pm 21)$ & $\dfrac{\sqrt{3}}{16}%
(17\pm 15)$ \\ \hline
$-\dfrac{1}{16}$ & $\dfrac{7\sqrt{3}}{36}$ & $\dfrac{17\sqrt{3}}{12}$ & $%
\dfrac{1}{32}(1\pm \sqrt{17})$ & $\dfrac{17\sqrt{51}\pm 65\sqrt{3}}{288}$ & $%
\dfrac{9\sqrt{51}\pm 17\sqrt{3}}{48}$ \\ \hline
$-\dfrac{1}{80}$ & $\dfrac{\sqrt{15}}{12}$ & $\dfrac{3\sqrt{15}}{4}$ & $%
\dfrac{1}{160}(1\pm 9)$ & $\dfrac{\sqrt{3}}{96}(19\pm 11)$ & $\dfrac{\sqrt{3}%
}{48}(41\pm 9)$ \\ \hline
$-\dfrac{1}{1024}$ & $\dfrac{53\sqrt{3}}{288}$ & $\dfrac{205\sqrt{3}}{96}$ & 
$\dfrac{1}{2048}(1\pm 5\sqrt{41})$ & $\dfrac{533\sqrt{123}\pm 721\sqrt{3}}{%
18432}$ & $\dfrac{513\sqrt{123}\pm 205\sqrt{3}}{3072}$ \\ \hline
$-\dfrac{1}{3024}$ & $\dfrac{13\sqrt{7}}{108}$ & $\dfrac{55\sqrt{7}}{36}$ & $%
\dfrac{1}{6048}(1\pm 55)$ & $\dfrac{\sqrt{3}}{7776}(1433\pm 191)$ & $\dfrac{%
\sqrt{3}}{1296}(1513\pm 55)$ \\ \hline
$-\dfrac{1}{500^{2}}$ & $\dfrac{827\sqrt{3}}{4500}$ & $\dfrac{4717\sqrt{3}}{%
1500}$ & $\dfrac{1}{500000}(1\pm 53\sqrt{89})$ & $\dfrac{17533\sqrt{267}\pm
3161\sqrt{3}}{900000}$ & $\dfrac{125001\sqrt{267}\pm 4717\sqrt{3}}{750000}$
\\ \hline
\end{tabular}%
\]

$s=\dfrac{1}{4}\qquad $%
\[
\small
\begin{tabular}{|l|l|l|l|l|l|}
\hline
$x_{0}$ & $a$ & $b$ & $w_{0}$ & $A$ & $B$ \\ \hline
$-\dfrac{1}{4}$ & $\dfrac{3}{8}$ & $\dfrac{5}{2}$ & $\dfrac{1}{8}(1\pm \sqrt{%
5})$ & $\dfrac{\pm 13+5\sqrt{5}}{16}$ & $\dfrac{\pm 5+3\sqrt{5}}{4}$ \\ 
\hline
$-(\dfrac{16}{63})^{2}$ & $\dfrac{8\sqrt{7}}{63}$ & $\dfrac{65\sqrt{7}}{63}$
& $\dfrac{8}{63^{2}}(16\pm 65)$ & $\dfrac{8\sqrt{7}}{49}(1\pm 1)$ & $\dfrac{%
\sqrt{7}}{7038}(4481\pm 2080)$ \\ \hline
$-\dfrac{1}{48}$ & $\dfrac{\sqrt{3}}{16}$ & $\dfrac{7\sqrt{3}}{12}$ & $%
\dfrac{1}{96}(1\pm 7)$ & $\dfrac{1}{192}(23\pm 17)$ & $\dfrac{1}{48}(25\pm
7) $ \\ \hline
$-\dfrac{1}{324}$ & $\dfrac{23}{72}$ & $\dfrac{65}{18}$ & $\dfrac{1}{648}%
(1\pm 5\sqrt{13})$ & $\dfrac{\pm 17+13\sqrt{13}}{144}$ & $\dfrac{\pm 65+163%
\sqrt{13}}{324}$ \\ \hline
$-\dfrac{1}{5\cdot 72^{2}}$ & $\dfrac{41\sqrt{5}}{288}$ & $\dfrac{161\sqrt{5}%
}{72}$ & $\dfrac{1}{51840}(1\pm 161)$ & $\dfrac{1}{6912}(2201\pm 121)$ & $%
\dfrac{1}{5184}(12961\pm 161)$ \\ \hline
$-\dfrac{1}{882^{2}}$ & $\dfrac{1123}{3528}$ & $\dfrac{5365}{882}$ & $\dfrac{%
1}{1555848}(1\pm 145\sqrt{37})$ & $\dfrac{\pm 439+6031\sqrt{37}}{115248}$ & $%
\dfrac{\pm 5365+388963\sqrt{37}}{777924}$ \\ \hline
\end{tabular}%
\]

$s=\dfrac{1}{6}$\\[2mm]
\hspace*{-20mm}
\small
\begin{tabular}{|l|l|l|l|l|l|}
\hline
$x_{0}$ & $a$ & $b$ & $w_{0}$ & $A$ & $B$ \\ \hline
$-(\dfrac{4}{5})^{3}$ & $\dfrac{8\sqrt{15}}{75}$ & $\dfrac{21\sqrt{15}}{25}$
& $\dfrac{32\pm 12\sqrt{21}}{125}$ & $\dfrac{168\sqrt{7}\pm 316\sqrt{3}}{375}
$ & $\dfrac{253\sqrt{7}\pm 336\sqrt{3}}{250}$ \\ \hline
$-(\dfrac{3}{8})^{3}$ & $\dfrac{15\sqrt{2}}{64}$ & $\dfrac{77\sqrt{2}}{32}$
& $\dfrac{27\pm 21\sqrt{33}}{1024}$ & $\dfrac{33\sqrt{11}\pm 69\sqrt{3}}{256}
$ & $\dfrac{283\sqrt{11}\pm 231\sqrt{3}}{512}$ \\ \hline
$-\dfrac{1}{512}$ & $\dfrac{25\sqrt{6}}{192}$ & $\dfrac{57\sqrt{6}}{32}$ & $%
\dfrac{1\pm 3\sqrt{57}}{1024}$ & $\dfrac{57\sqrt{19}\pm 49\sqrt{3}}{768}$ & $%
\dfrac{257\sqrt{19}\pm 57\sqrt{3}}{512}$ \\ \hline
$-\dfrac{9}{40^{3}}$ & $\dfrac{93\sqrt{30}}{1600}$ & $\dfrac{759\sqrt{30}}{%
800}$ & $\dfrac{9\pm 759}{128000}$ & $\dfrac{\sqrt{3}(5889\pm 639)}{32000}$
& $\dfrac{\sqrt{3}(96027\pm 2277)}{64000}$ \\ \hline
$-\dfrac{1}{80^{3}}$ & $\dfrac{263\sqrt{15}}{3200}$ & $\dfrac{2709\sqrt{15}}{%
1600}$ & $\dfrac{1\pm 63\sqrt{129}}{1024000}$ & $\dfrac{12427\sqrt{43}\pm 743%
\sqrt{3}}{256000}$ & $\dfrac{256001\sqrt{43}\pm 2709\sqrt{3}}{512000}$ \\ 
\hline
$-\dfrac{1}{440^{3}}$ & $\dfrac{10177\sqrt{330}}{580800}$ & $\dfrac{43617%
\sqrt{330}}{96800}$ & $\dfrac{1\pm 651\sqrt{201}}{170368000}$ & $\dfrac{%
4968921\sqrt{67}\pm 35257\sqrt{3}}{127776000}$ & $\dfrac{42592001\sqrt{67}%
\pm 43617\sqrt{3}}{85184000}$ \\ \hline
$-\frac{1}{53360^{3}}$ & $\frac{13591409\sqrt{10005}}{227897059584000}$ & $%
\frac{90856689\sqrt{10005}}{37982843264000}$ & $\frac{1\pm 557403\sqrt{489}}{%
303862746112000}$ & $\frac{5681919113121\sqrt{163}\pm 71540369\sqrt{3}}{%
12160587099402240000}$ & $\frac{75965686528001\sqrt{163}\pm 90856689\sqrt{3}%
}{8107058066268160000}$ \\ \hline
\end{tabular}\\[2mm]

This takes care of formulas 4.2-4.13 except 4.7 which comes from a divergent
series with \ $x_{0}=-\dfrac{16}{9}$. Note that we find a new formula with
rational \ $w_{0}$ for \ $s=\dfrac{1}{6}.$

\textbf{Remark: }Formula (4.11) in [6] is false. The right hand side should
be \ $\dfrac{162\sqrt{7}}{343\pi }.$

Formula 4.1 is of different kind. It is a special case of

\pagebreak

\textbf{Proposition 5.}

We have%
\[
\sum_{n=0}^{\infty }\sum_{k=0}^{n}\dbinom{-s}{k}^{2}\dbinom{-(1-s)}{n-k}%
^{2}x^{n}=\frac{1}{1-x}F(1/2,s,1-s;1,1;-\frac{4x}{(1-x)^{2}}) 
\]

\textbf{Classical Proof:}

The left hand side is%
\[
L=\sum_{n=0}^{\infty }\sum_{k=0}^{n}\frac{(s)_{k}^{2}}{k!^{2}}\frac{%
(1-s)_{n-k}^{2}}{(n-k)!^{2}}x^{n}=F(s,s;,1;x)F(1-s,1-s;1;x) 
\]%
\[
=\frac{1}{1-x}F(s,1-s;1;\frac{x}{x-1})^{2} 
\]%
after using Pfaff's identity twice%
\[
F(a,b;c;x)=(1-x)^{-a}F(a,c-b;c;\frac{x}{x-1}) 
\]%
Now we use 
\[
F(2a,2b;a+b+1/2;x)=F(a,b;a+b+1/2;4x(1-x)) 
\]%
again to get%
\[
L=\frac{1}{1-x}F(s/2,(1-s)/2;1;-\frac{4x}{(1-x)^{2}})^{2} 
\]%
\[
=\frac{1}{1-x}F(1/2,s,1-s;1,1;-\frac{4x}{(1-x)^{2}}) 
\]%
by Clausen's identity.

\textbf{Maple Proof:}

Both sides satisfy%
\[
y^{\prime \prime \prime }+\frac{3(2-5x)}{x(1-x)}y^{\prime \prime }+\frac{%
1-(10+s-s^{2})x+(12+s-s^{2})x^{2}}{x^{2}(1-x)^{2}}y^{\prime }-\frac{1}{2}%
\frac{2+s-s^{2}-(6+3s-3s^{2})x)}{x^{2}(1-x)^{2}}y=0 
\]%
Then we check the first terms in the power series.

Let 
\[
a_{n}=\frac{(1/2)_{n}(s)_{n}(1-s)_{n}}{n!^{3}} 
\]%
and%
\[
A_{n}=\sum_{k=0}^{n}\dbinom{-s}{k}^{2}\dbinom{-(1-s)}{n-k}^{2} 
\]%
Then copying the proof of Proposition 4 we get for every formula%
\[
\sum_{n=0}^{\infty }a_{n}(a+bn)x_{0}^{n}=\frac{1}{\pi } 
\]%
a new formula%
\[
\sum_{n=0}^{\infty }A_{n}(A+Bn)w_{0}^{n}=\frac{1}{\pi } 
\]%
where%
\[
w_{0}=1-\frac{2}{x_{0}}(1-\sqrt{1-x_{0}}) 
\]%
\[
A=(1-w_{0})(a-\frac{bw_{0}}{1+w_{0}}) 
\]%
\[
B=\frac{b(1-w_{0})^{2}}{1+w_{0}} 
\]%
We give only the rational \ $w_{0}$%
\[
\begin{tabular}{|l|l|l|l|l|l|l|}
\hline
$s$ & $x_{0}$ & $a$ & $b$ & $w_{0}$ & $A$ & $B$ \\ \hline
$\dfrac{1}{3}$ & $-\dfrac{9}{16}$ & $\dfrac{\sqrt{3}}{4}$ & $\dfrac{5\sqrt{3}%
}{4}$ & $\dfrac{1}{9}$ & $\dfrac{\sqrt{3}}{9}$ & $\dfrac{8\sqrt{3}}{9}$ \\ 
\hline
$\dfrac{1}{4}$ & $-(\dfrac{16}{63})^{2}$ & $\dfrac{8\sqrt{7}}{63}$ & $\dfrac{%
65\sqrt{7}}{63}$ & $\dfrac{1}{64}$ & $\dfrac{7\sqrt{7}}{64}$ & $\dfrac{63%
\sqrt{7}}{64}$ \\ \hline
$\dfrac{1}{4}$ & $\dfrac{32}{81}$ & $\dfrac{2}{9}$ & $\dfrac{14}{9}$ & $-%
\dfrac{1}{8}$ & $\dfrac{1}{2}$ & $\dfrac{9}{4}$ \\ \hline
\end{tabular}%
\]%
\[
\]

The formulas (2.2)-(2.4) in [6] due to the twin brother Z.H.Sun are special
cases of the following

\textbf{Proposition 6: }We have%
\[
F(s,1-s;1;\frac{1}{2}(1-\sqrt{1-x}))^{2}=F(\frac{1}{2},s,1-s;1,1;x) 
\]

\textbf{Classical Proof:}

Solving for \ $x$ \ we have the equivalent statement%
\[
F(s,1-s;1;w)^{2}=F(\frac{1}{2},s,1-s;1,1;4w(1-w)) 
\]%
Using (formula 3.1.3, p.125 in [2])%
\[
F(2a,2b;,a+b+1/2;w)=F(a,b;a+b+1/2;4w(1-w)) 
\]%
we get

\[
F(s,1-s;1;w)^{2}=F(s/2,(1-s)/2;1;4w(1-w))^{2} 
\]%
and finish by Clausen's identity.

\textbf{Maple Proof:}

\textbf{\ }$\ $One verifies that both sides\ satisfy the differential
equation

\[
y^{\prime \prime \prime }+\frac{3}{2}\frac{2-3x}{x(1-x)}y^{\prime \prime }+%
\frac{1-(3+s-s^{2})x}{x^{2}(1-x)}y^{\prime }-\frac{1}{2}\frac{s(1-s)}{%
x^{2}(1-x)}y=0 
\]%
One expands both sides in power series and checks the first few coefficients.

Theorem 1.3 in [7] is a special case of the following transformation.

Let%
\[
A_{n}=\sum_{k=0}^{n}\frac{(s)_{k}(1-s)_{k}(s)_{n-k}(1-s)_{n-k}}{%
k!^{2}(n-k)!^{2}} 
\]%
so%
\[
F(s,1-s;1;x)^{2}=\sum_{n=0}^{\infty }A_{n}x^{n} 
\]%
so we have the following result.

\textbf{Proposition 7.}

Assume we have a formula%
\[
\sum_{n=0}^{\infty }\frac{(1/2)_{n}(s)_{n}(1-s)_{n}}{n!^{3}}(a+bn)x_{0}^{n}=%
\frac{1}{\pi } 
\]%
Then we have 
\[
\sum_{n=0}^{\infty }A_{n}(A+Bn)w_{0}^{n}=\frac{1}{\pi } 
\]%
where%
\[
w_{0}=\frac{1}{2}(1-\sqrt{1-x_{0}}) 
\]%
and%
\[
A=a 
\]%
\[
B=b\frac{1-w_{0}}{1-2w_{0}} 
\]

\textbf{Proof: }Let \ $\theta =x\dfrac{d}{dx}$ and \ $a_{n}=\dfrac{%
(1/2)_{n}(s)_{n}(1-s)_{n}}{n!^{3}}$. Then%
\[
\sum_{n=0}^{\infty }a_{n}(a+bn)x^{n}=(a+b\theta )\sum_{n=0}^{\infty
}a_{n}x^{n} 
\]%
\[
=(a+b\theta )\sum_{n=0}^{\infty }A_{n}(\frac{1}{2}(1-\sqrt{1-x}%
))^{n}=\sum_{n=0}^{\infty }A_{n}(a+\frac{bn}{2\sqrt{1-x}(1-\sqrt{1-x})})(%
\frac{1}{2}(1-\sqrt{1-x}))^{n} 
\]%
Putting $x=x_{0}$ we get%
\[
A=a 
\]%
and%
\[
B=\frac{b}{2\sqrt{1-x_{0}}(1-\sqrt{1-x_{0}})}=b\frac{1-w_{0}}{1-2w_{0}} 
\]

$s=\dfrac{1}{2}$

\[
\begin{tabular}{|l|l|l|l|l|}
\hline
$x_{0}$ & $w_{0}$ & $a=A$ & $b$ & $B$ \\ \hline
$\dfrac{1}{4}$ & $\dfrac{1}{2}-\dfrac{\sqrt{3}}{4}$ & $\dfrac{1}{4}$ & $%
\dfrac{3}{2}$ & $\dfrac{3+2\sqrt{3}}{4}$ \\ \hline
$\dfrac{1}{64}$ & $\dfrac{1}{2}-\dfrac{3\sqrt{7}}{16}$ & $\dfrac{5}{16}$ & $%
\dfrac{21}{8}$ & $\dfrac{21+8\sqrt{7}}{16}$ \\ \hline
$-1$ & $\dfrac{1}{2}-\dfrac{\sqrt{2}}{2}$ & $\dfrac{1}{2}$ & $2$ & $\dfrac{2+%
\sqrt{2}}{2}$ \\ \hline
$-\dfrac{1}{8}$ & $\dfrac{1}{2}-\dfrac{3\sqrt{2}}{8}$ & $\dfrac{\sqrt{2}}{4}$
& $\frac{3\sqrt{2}}{2}$ & $\dfrac{4+3\sqrt{2}}{4}$ \\ \hline
\end{tabular}%
\]

$s=\dfrac{1}{3}$%
\[
\begin{tabular}{|l|l|l|l|l|}
\hline
$x_{0}$ & $w_{0}$ & $a=A$ & $b$ & $B$ \\ \hline
$\frac{1}{2}$ & $\dfrac{1}{2}-\dfrac{\sqrt{2}}{4}$ & $\dfrac{\sqrt{3}}{9}$ & 
$\dfrac{2\sqrt{3}}{3}$ & $\dfrac{\sqrt{3}+\sqrt{6}}{9}$ \\ \hline
$\dfrac{2}{27}$ & $\dfrac{1}{2}-\dfrac{5\sqrt{3}}{18}$ & $\dfrac{8}{27}$ & $%
\dfrac{20}{9}$ & $\dfrac{10+6\sqrt{3}}{9}$ \\ \hline
$\dfrac{4}{125}$ & $\dfrac{1}{2}-\dfrac{11\sqrt{5}}{50}$ & $\dfrac{8\sqrt{3}%
}{45}$ & $\dfrac{22\sqrt{3}}{45}$ & $\dfrac{11\sqrt{3}+5\sqrt{15}}{15}$ \\ 
\hline
$-\dfrac{9}{16}$ & $-\dfrac{1}{8}$ & $\dfrac{\sqrt{3}}{4}$ & $\dfrac{5\sqrt{3%
}}{4}$ & $\dfrac{9\sqrt{3}}{8}$ \\ \hline
$-\dfrac{1}{16}$ & $\dfrac{1}{2}-\dfrac{\sqrt{17}}{8}$ & $\dfrac{7\sqrt{3}}{%
36}$ & $\dfrac{17\sqrt{3}}{12}$ & $\dfrac{17\sqrt{3}+4\sqrt{51}}{24}$ \\ 
\hline
$-\dfrac{1}{80}$ & $\dfrac{1}{2}-\dfrac{9\sqrt{5}}{40}$ & $\dfrac{\sqrt{15}}{%
12}$ & $\dfrac{3\sqrt{15}}{4}$ & $\dfrac{20\sqrt{3}+9\sqrt{15}}{24}$ \\ 
\hline
$-\dfrac{1}{1024}$ & $\dfrac{1}{2}-\dfrac{5\sqrt{41}}{64}$ & $\dfrac{53\sqrt{%
3}}{288}$ & $\dfrac{205\sqrt{3}}{96}$ & $\dfrac{205\sqrt{3}+32\sqrt{123}}{192%
}$ \\ \hline
$-\dfrac{1}{3024}$ & $\dfrac{1}{2}-\dfrac{55\sqrt{21}}{504}$ & $\dfrac{13%
\sqrt{7}}{108}$ & $\dfrac{55\sqrt{7}}{36}$ & $\dfrac{84\sqrt{3}+55\sqrt{7}}{%
72}$ \\ \hline
$-\dfrac{1}{500^{2}}$ & $\dfrac{1}{2}-\dfrac{53\sqrt{89}}{1000}$ & $\dfrac{%
827\sqrt{3}}{4500}$ & $\dfrac{4717\sqrt{3}}{1500}$ & $\dfrac{4717\sqrt{3}+500%
\sqrt{267}}{3000}$ \\ \hline
\end{tabular}%
\]

$s=\dfrac{1}{4}$%
\[
\begin{tabular}{|l|l|l|l|l|}
\hline
$x_{0}$ & $w_{0}$ & $a=A$ & $b$ & $B$ \\ \hline
$\dfrac{32}{81}$ & $\dfrac{1}{9}$ & $\dfrac{2}{9}$ & $\dfrac{14}{9}$ & $%
\dfrac{16}{9}$ \\ \hline
$\dfrac{1}{9}$ & $\dfrac{1}{2}-\dfrac{\sqrt{2}}{3}$ & $\dfrac{\sqrt{3}}{6}$
& $\dfrac{4\sqrt{3}}{3}$ & $\dfrac{4\sqrt{3}+3\sqrt{6}}{6}$ \\ \hline
$\dfrac{1}{81}$ & $\dfrac{1}{2}-\dfrac{2\sqrt{5}}{9}$ & $\dfrac{2\sqrt{2}}{9}
$ & $\dfrac{20\sqrt{2}}{9}$ & $\dfrac{20\sqrt{2}+9\sqrt{10}}{18}$ \\ \hline
$\dfrac{1}{49^{2}}$ & $\dfrac{1}{2}-\dfrac{10\sqrt{2}}{49}$ & $\dfrac{9\sqrt{%
3}}{49}$ & $\dfrac{120\sqrt{3}}{49}$ & $\dfrac{147\sqrt{2}+120\sqrt{3}}{98}$
\\ \hline
$\dfrac{1}{99^{2}}$ & $\dfrac{1}{2}-\dfrac{35\sqrt{2}}{99}$ & $\dfrac{19%
\sqrt{11}}{198}$ & $\dfrac{140\sqrt{11}}{99}$ & $\dfrac{140\sqrt{11}+99\sqrt{%
22}}{198}$ \\ \hline
$\dfrac{1}{99^{4}}$ & $\dfrac{1}{2}-\dfrac{910\sqrt{29}}{9801}$ & $\dfrac{%
2206\sqrt{2}}{9801}$ & $\dfrac{52780\sqrt{2}}{9801}$ & $\dfrac{52780\sqrt{2}%
+9801\sqrt{58}}{19602}$ \\ \hline
$-\dfrac{1}{4}$ & $\dfrac{1}{2}-\dfrac{\sqrt{5}}{4}$ & $\dfrac{3}{8}$ & $%
\dfrac{5}{2}$ & $\dfrac{5+2\sqrt{5}}{4}$ \\ \hline
$-(\dfrac{16}{63})^{2}$ & $-\dfrac{1}{63}$ & $\dfrac{8\sqrt{7}}{63}$ & $%
\dfrac{65\sqrt{7}}{63}$ & $\dfrac{64\sqrt{7}}{63}$ \\ \hline
$-\dfrac{1}{48}$ & $\dfrac{1}{2}-\dfrac{7\sqrt{3}}{24}$ & $\dfrac{3\sqrt{3}}{%
16}$ & $\dfrac{7\sqrt{3}}{4}$ & $\dfrac{12+7\sqrt{3}}{8}$ \\ \hline
$-\dfrac{1}{324}$ & $\dfrac{1}{2}-\dfrac{5\sqrt{13}}{36}$ & $\dfrac{23}{72}$
& $\dfrac{65}{18}$ & $\dfrac{65+18\sqrt{13}}{36}$ \\ \hline
$-\dfrac{1}{5\cdot 72^{2}}$ & $\dfrac{1}{2}-\dfrac{161\sqrt{5}}{720}$ & $%
\dfrac{41\sqrt{5}}{288}$ & $\dfrac{161\sqrt{5}}{72}$ & $\dfrac{360+161\sqrt{5%
}}{144}$ \\ \hline
$-\dfrac{1}{882^{2}}$ & $\dfrac{1}{2}-\dfrac{145\sqrt{37}}{1764}$ & $\dfrac{%
1123}{3528}$ & $\dfrac{5365}{882}$ & $\dfrac{5365+882\sqrt{37}}{1764}$ \\ 
\hline
\end{tabular}%
\]

$s=\dfrac{1}{6}$%
\[
\begin{tabular}{|l|l|l|l|l|}
\hline
$x_{0}$ & $w_{0}$ & $a=A$ & $b$ & $B$ \\ \hline
$\dfrac{27}{125}$ & $\dfrac{1}{2}-\dfrac{7\sqrt{10}}{50}$ & $\dfrac{3\sqrt{5}%
}{25}$ & $\dfrac{28\sqrt{5}}{25}$ & $\dfrac{25\sqrt{2}+14\sqrt{5}}{25}$ \\ 
\hline
$\dfrac{4}{125}$ & $\dfrac{1}{2}-\dfrac{11\sqrt{5}}{50}$ & $\dfrac{2\sqrt{15}%
}{25}$ & $\dfrac{22\sqrt{15}}{25}$ & $\dfrac{25\sqrt{3}+11\sqrt{15}}{25}$ \\ 
\hline
$\dfrac{8}{11^{3}}$ & $\dfrac{1}{2}-\dfrac{21\sqrt{33}}{242}$ & $\dfrac{20%
\sqrt{33}}{363}$ & $\dfrac{84\sqrt{33}}{121}$ & $2+\dfrac{42\sqrt{33}}{121}$
\\ \hline
$\dfrac{64}{85^{3}}$ & $\dfrac{1}{2}-\dfrac{171\sqrt{1785}}{14450}$ & $%
\dfrac{144\sqrt{255}}{7225}$ & $\dfrac{2394\sqrt{255}}{7225}$ & $\sqrt{7}+%
\dfrac{1197\sqrt{255}}{7225}$ \\ \hline
$-(\dfrac{4}{5})^{3}$ & $\dfrac{1}{2}-\dfrac{3\sqrt{105}}{50}$ & $\dfrac{8%
\sqrt{15}}{75}$ & $\dfrac{21\sqrt{15}}{25}$ & $\dfrac{25\sqrt{7}+21\sqrt{15}%
}{50}$ \\ \hline
$-(\dfrac{3}{8})^{3}$ & $\dfrac{1}{2}-\dfrac{7\sqrt{22}}{64}$ & $\dfrac{15%
\sqrt{2}}{64}$ & $\dfrac{77\sqrt{2}}{32}$ & $\dfrac{77\sqrt{2}+32\sqrt{11}}{%
64}$ \\ \hline
$-\dfrac{1}{8^{3}}$ & $\dfrac{1}{2}-\dfrac{3\sqrt{114}}{64}$ & $\dfrac{25%
\sqrt{6}}{192}$ & $\dfrac{57\sqrt{6}}{32}$ & $\dfrac{57\sqrt{6}+32\sqrt{19}}{%
64}$ \\ \hline
$-\dfrac{9}{40^{3}}$ & $\dfrac{1}{2}-\dfrac{253\sqrt{10}}{1600}$ & $\dfrac{93%
\sqrt{30}}{1600}$ & $\dfrac{759\sqrt{30}}{800}$ & $\dfrac{2400\sqrt{3}+759%
\sqrt{30}}{1600}$ \\ \hline
$-\dfrac{1}{80^{3}}$ & $\dfrac{1}{2}-\dfrac{63\sqrt{645}}{3200}$ & $\dfrac{%
263\sqrt{15}}{3200}$ & $\dfrac{2709\sqrt{15}}{1600}$ & $\dfrac{2709\sqrt{15}%
+1600\sqrt{43}}{3200}$ \\ \hline
$-\dfrac{1}{440^{3}}$ & $\dfrac{1}{2}-\dfrac{651\sqrt{22110}}{193600}$ & $%
\dfrac{10177\sqrt{330}}{580800}$ & $\dfrac{43617\sqrt{330}}{96800}$ & $%
\dfrac{96800\sqrt{67}+43617\sqrt{330}}{193600}$ \\ \hline
$-\dfrac{1}{53360^{3}}$ & $\dfrac{1}{2}-\dfrac{651\sqrt{22110}}{193600}$ & $%
\dfrac{13591409\sqrt{10005}}{227897059584000}$ & $\dfrac{90856689\sqrt{10005}%
}{37982843264000}$ & see below \\ \hline
\end{tabular}%
\]

In the last row 
\[
B=\dfrac{711822400\sqrt{163}+90856689\sqrt{10005}}{75965686528000} 
\]

\textbf{Remark: }When \ $x_{0}$ \ is positive then we get a (slowly)
convergent series with \ $\frac{1}{2}(1+\sqrt{1-x_{0}})$ \ but the sum is
not \ $\dfrac{1}{\pi }$ \ (rather a negative multiple of it ). Why?

\bigskip

\textbf{References.}

\textbf{1. }G.Almkvist, Str\"{a}ngar i m\aa nsken, Normat, 51 (2003), 22-33.

\textbf{2. }G.E.Andrews.R.Askey,R.Roy, Special Functions, Cambridge
University Press, 1999.

\textbf{3. }H.H.Chan, J.Wan, W.Zudilin, Legendre polynomials and
Ramanujan-type series for $\dfrac{1}{\pi },$

\textbf{4. }J.Guillera, Tables of Ramanujan series with rational values of $%
z $, Guillera's home page

\textbf{5. }J.Guillera, WZ-proofs of "divergent" Ramanujan-type series,
NT/1012.2681.

\textbf{6. }Z.W.Sun, List of conjectural series for powers of $\pi $ and
other constants, CA/1102.5649

\textbf{7. }Z.W.Sun, Some new series for $\dfrac{1}{\pi }$ and related
congruences, NT/1104.3856.

\textbf{8. }Z.W.Sun, Supercongruences and Eulernumbers, Sci. China Math. 54
(2011), 2509-2535.

\textbf{9. }J.Wan, W.Zudilin, Generating functions of Legendre polynomials:
A tribute to Fred Brafman,

\[
\]

\bigskip Institute of Algebraic Meditation \ \ \ \ \ \ \ \ \ \ \ \ \ \ \ \ \
\ \ \ \ \ Johannes Gutenberg-Universit\"{a}t

Fogdar\"{o}d 208, H\"{o}\"{o}r, S24333 Sweden \ \ \ \ \ \ \ \ \ \ \ \ \
D-55099 Mainz, Germany\ 

gert.almkvist@yahoo.se \ \ \ \ \ \ \ \ \ \ \ \ \ \ \ \ \ \ \ \ \ \ \ \ \ \ \
\ \ \ \ \ \ \ \ black\_Dr.House@gmx.de

\bigskip

\[
\text{\bf Appendix: A class of slowly converging series for $1/\pi$.}
\]
\[
\text{Arne Meurman}
\]

In the final remark Almkvist and Aycock ask why, when one considers the power 
series at $w_1 = \tfrac12(1+\sqrt{1-x_0})$, instead of at 
$w_0= \tfrac12(1-\sqrt{1-x_0})$, one gets formulas for negative multiples of 
$\frac{1}{\pi}$. Here we shall prove such formulas in the 
cases $s=1/2,1/3,1/4,1/6$ in Proposition 7.

Following \cite{CWZ} we set
\[
F(t) = F(s,t) = {}_2F_1(s,1-s;1;t),
\]
\[
G(t) = t\frac{dF}{dt},
\]
and let for $s=1/2,1/3,1/4,1/6$\ \  $t(\tau)=t_N(\tau)$ be given by
\[
t_4(\tau) = \left(1+\frac{1}{16}\left(\frac{\eta(\tau)}
{\eta(4\tau)}\right)^8\right)^{-1},\quad
t_3(\tau) = \left(1+\frac{1}{27}\left(\frac{\eta(\tau)}
{\eta(3\tau)}\right)^{12}\right)^{-1},
\]
\[
t_2(\tau) = \left(1+\frac{1}{64}\left(\frac{\eta(\tau)}
{\eta(2\tau)}\right)^{24}\right)^{-1},\quad
t_1(\tau) = \frac{1}{2}-\frac{1}{2}\sqrt{1-\frac{1728}{j(\tau)}}.
\]

Let $U$ be the connected component of $\{\tau\in{\bf C}\mid \Im(\tau)>0,
|t(\tau)|<1\}$ which contains all $\tau$ with sufficiently large 
imaginary part,
 a "neighborhood of $i\infty$".
Let $\tau_0\in U$ such that
\begin{equation}
w_0=t(\tau_0),\quad w_1=1-t(\tau_0)
\end{equation}
satisfy
\begin{equation}
|w_0|<1,\quad |w_1|<1.
\end{equation}
In the Ramanujan-type formulas, $\tau_0$ is usually a quadratic 
irrationality. Let $A_n$ be defined by the power series expansions
\begin{equation}
F^2(t) = \sum_{n=0}^\infty A_nt^n,
\end{equation}
as in Proposition 7. Set
\begin{equation}
C_s = \frac{1}{2\sin(\pi s)}.
\end{equation}

\begin{theorem}
Assume that there is an identity
\begin{equation}
\sum_{n=0}^\infty (A+Bn)A_nw_0^n = \frac{C}{\pi},
\end{equation}
equivalently
\begin{equation}
AF^2(w_0) +2BF(w_0)G(w_0) = \frac{C}{\pi}.
\label{FG1}
\end{equation}
Then
\begin{equation}
\sum_{n=0}^\infty (\hat A+\hat Bn)A_nw_1^n 
= \frac{\hat C}{\pi},
\label{w1id}
\end{equation}
where
\begin{equation}
\begin{split}
\hat A &= A,\\
\hat B &= -B\frac{w_0}{w_1},\\
\hat C &= \frac{C\left(\frac{\tau_0}{i}\right)^2}{C_s^2} 
- \frac{B\left(\frac{\tau_0}{i}\right)}{C_s^2w_1}.
\end{split}
\end{equation}
\end{theorem}
\begin{proof}
By formulas (8), (9) in \cite{CWZ} we have, for $\tau\in U$,
\begin{equation}
\tau = iC_s\frac{F(1-t)}{F(t)},
\label{fFricke}
\end{equation}
and
\begin{equation}
\frac{1}{2\pi i}\frac{dt}{d\tau} = q\frac{dt}{dq} = t(1-t)F^2(t),
\label{difft}
\end{equation}
where $q = e^{2\pi i\tau}$.
Take $\frac{1}{2\pi i}$ times the logarithmic derivative with respect 
to $\tau$ in \eqref{fFricke}:
\[
\frac{\frac{dF}{dt}(1-t)\cdot(-1)\cdot\frac{1}{2\pi i}\frac{dt}
{d\tau}}{F(1-t)} - \frac{\frac{dF}{dt}\frac{1}{2\pi i}\frac{dt}
{d\tau}}{F(t)} = \frac{1}{2\pi i\tau}.
\]
Substitute \eqref{difft} to obtain
\[
-\frac{G(1-t)tF^2(t)}{F(1-t)} -G(t)(1-t)F(t) = \frac{1}{2\pi i\tau}.
\]
Multiply by $2i\tau$ and substitute \eqref{fFricke}:
\begin{equation}
2C_s\,t\,G(1-t)F(t)+2\left(\frac{\tau}{i}\right)(1-t)G(t)F(t) = \frac{1}{\pi}.
\label{FGgen}
\end{equation}
Evaluate \eqref{FGgen} at $\tau = \tau_0$:
\begin{equation}
2C_s\,w_0\,G(w_1)F(w_0)+2\left(\frac{\tau_0}{i}\right)w_1G(w_0)F(w_0) 
= \frac{1}{\pi}.
\label{FG2}
\end{equation}
By assumption \eqref{FG1}
\begin{equation}
2BF(w_0)G(w_0) = \frac{C}{\pi} - AF^2(w_0),
\label{FG0}
\end{equation}
and eliminating $F(w_0)G(w_0)$ by \eqref{FG2}, \eqref{FG0} we have
\[
2BC_s\,w_0\,G(w_1)F(w_0)-A\left(\frac{\tau_0}{i}\right)w_1F^2(w_0) 
= \frac{1}{\pi}\left(B-C\left(\frac{\tau_0}{i}\right)w_1\right).
\]
Substitute $F(w_0)=\left(\frac{i}{\tau_0}\right)C_sF(w_1)$ from 
\eqref{fFricke} to obtain
\[
-A\left(\frac{i}{\tau_0}\right)C_s^2w_1F^2(w_1) 
+ 2B\left(\frac{i}{\tau_0}\right)C_s^2\,w_0\,F(w_1)G(w_1)
 = \frac{1}{\pi}\left(B-C\left(\frac{\tau_0}{i}\right)w_1\right).
\]
Dividing by $-\left(\frac{i}{\tau_0}\right)C_s^2w_1$  we get
\[
AF^2(w_1) - 2B\frac{w_0}{w_1}F(w_1)G(w_1)
 = \frac{1}{\pi}\left(\frac{C\left(\frac{\tau_0}{i}\right)^2}{C_s^2} 
- \frac{B\left(\frac{\tau_0}{i}\right)}{C_s^2w_1}\right),
\]
which is equivalent to \eqref{w1id}.
\end{proof}

{\bf Remark.} The arguments in the proof above are analogous to some 
arguments in the proof of \cite{CCL} Theorem 2.1.

Theorem 1 applies to all of the identities in the Tables 
following Proposition 7 where $x_0>0$. We present 4 examples of such.

\bigskip
{\bf {Example 1.}} In case $s=\frac14, \tau_0=i$ we have $w_0 = t_2(i) 
= \frac19$, and there is in \cite{Sun2}, (1.12) the identity
\begin{equation}
\sum_{n=0}^\infty (1+8n)A_n\frac{1}{9^n} = \frac{9}{2\pi},
\end{equation}
where
\begin{equation}
A_n = \frac{1}{64^n}\sum_{k=0}^n \binom{2k}{k}\binom{4k}{2k}\binom{2(n-k)}{n-k}
\binom{4(n-k)}{2(n-k)}.
\label{An4}
\end{equation}
 In this case 
\[
A=1,\quad B=8,\quad C=\frac92, \quad C_s = \frac{1}{\sqrt{2}},
\]
and we obtain
\[
w_1 = \frac89,
\]
\[
\hat A = 1,\quad \hat B = -1,\quad \hat C = -9, 
\]
\begin{equation}
\sum_{n=0}^\infty (1-n)A_n\left(\frac{8}{9}\right)^n = -\frac{9}{\pi},
\end{equation}
which proves \cite{Sun1}, (2.10).

\bigskip
{\bf {Example 2.}} In case $s=\frac14, \tau_0=\frac{\sqrt{58}}{2}i$ we have
\[
w_0 = t_2\left(\frac{\sqrt{58}}{2}i\right) 
= \frac12-\frac{910}{9801}\sqrt{29},
\]
and there is in Proposition 7, Table $s=\frac14$, the identity
\begin{equation}
\sum_{n=0}^\infty \left(\frac{2206\sqrt{2}}{9801}
+\frac{52780\sqrt{2}+9801\sqrt{58}}{19602}n\right)A_nw_0^n 
=  \frac{1}{\pi},
\end{equation}
where $A_n$ is as in \eqref{An4}.
In this case 
\[
C_s = \frac{1}{\sqrt{2}},\quad w_1 = \frac12 + \frac{910}{9801}\sqrt{29},
\]
and we obtain
\begin{equation}
\sum_{n=0}^\infty \left(\frac{2206\sqrt{2}}{9801}+\frac{52780\sqrt{2}
-9801\sqrt{58}}{19602}n\right)A_nw_1^n = -\frac{29}{\pi}.
\end{equation}

\bigskip
{\bf {Example 3.}} In case $s=\frac12, \tau_0=\frac{\sqrt{3}}{2}i$ we have
\[
w_0 = t_4\left(\frac{\sqrt{3}}{2}i\right) 
= \frac12-\frac{\sqrt{3}}{4},
\]
and there is in Proposition 7, Table $s=\frac12$, the identity
\begin{equation}
\sum_{n=0}^\infty \left(\frac14+\frac{3+2\sqrt{3}}{4}n\right)A_nw_0^n 
 = \frac{1}{\pi},
\end{equation}
where
\begin{equation}
A_n = \frac{1}{16^n}\sum_{k=0}^n \binom{2k}{k}^2\binom{2(n-k)}{n-k}^2.
\end{equation}
In this case 
\[
C_s = \frac{1}{2},\quad w_1 = \frac12 + \frac{\sqrt{3}}{4},
\]
and we obtain
\begin{equation}
\sum_{n=0}^\infty \left(\frac14+\frac{3-2\sqrt{3}}{4}n\right)A_nw_1^n 
= -\frac{3}{\pi}.
\end{equation}

\bigskip
{\bf {Example 4.}} In case $s=\frac16, \tau_0=\sqrt{7}i$ we have
\[
w_0 = t_1\left(\sqrt{7}i\right) 
= \frac12-\frac{171}{14450}\sqrt{1785},
\]
and there is in Proposition 7, Table $s=\frac16$ the identity
\begin{equation}
\sum_{n=0}^\infty (A+Bn)A_nw_0^n =  \frac{1}{\pi},
\end{equation}
where
\begin{equation}
\begin{split}
&A = \frac{144\sqrt{255}}{7225},\quad B 
= \sqrt{7} + \frac{1197\sqrt{255}}{7225},\\
&A_n = \frac{1}{432^n}\sum_{k=0}^n \binom{3k}{k}\binom{6k}{3k}
\binom{3(n-k)}{n-k}\binom{6(n-k)}{3(n-k)}.
\end{split}
\end{equation}
In this case 
\[
C_s = 1,\quad w_1 = \frac12+\frac{171}{14450}\sqrt{1785},
\]
and we obtain
\begin{equation}
\sum_{n=0}^\infty (\hat A + \hat B n)A_nw_1^n = -\frac{7}{\pi},
\end{equation}
where
\[
\hat A = A,\quad \hat B = -\sqrt{7} + \frac{1197\sqrt{255}}{7225}.
\]

\noindent
Centre for Mathematical Sciences\newline
Mathematics\newline
Lund University\newline
Box 118\newline
SE-22100 Lund\newline
Sweden

\medskip
\noindent
arnem@maths.lth.se

\end{document}